\documentclass[12pt]{amsart}
\usepackage{amsfonts}
\usepackage{amssymb}
\usepackage{color}
\makeatletter \@addtoreset{equation}{section}

\makeatother \makeatletter \normalbaselineskip=14pt
\normalbaselines \headsep10mm

\newcommand{\mb}{\mathbb}
\newcommand{\mc}{\mathcal}
\newcommand{\eul}{\mathfrak}

\newcommand{\A}{\eul A}
\newcommand{\Ao}{{\eul A}_{\scriptscriptstyle 0}}

\newcommand{\vp}{\varphi}

\newcommand{\M}{{\mathfrak M}}
\newcommand{\T}{{\mc T}}
\newcommand{\Hil}{{\mc H}}
\newcommand{\C}{\mathrm{C}^{\ast}}

\newcommand{\mult}{\,{\scriptstyle \square}\,}
\newcommand{\D}{{\mc D}}

\def\x{\relax\ifmmode {\mbox{*}}\else*\fi}

\newcommand{\id}{e}
\newcommand{\ip}[2]{\langle{#1}|{#2}\rangle}

\newtheorem{defn}{Definition}[section]
\newtheorem{prop}[defn]{Proposition}
\newtheorem{thm}[defn]{Theorem}
\newtheorem{lemma}[defn]{Lemma}
\newtheorem{cor}[defn]{Corollary}
\theoremstyle{remark}
\newtheorem{rem}[defn]{Remark}
\newtheorem{example}[defn]{Example}
\newcommand{\bedefi}{\begin{defn}$\!\!${\bf }$\;$\rm }
\newcommand{\findefi}{ \end{defn}}

\newcommand{\ad}{^{\mbox{\scriptsize $\dag$}}}
\newcommand{\LDH}{{\mathcal L}\ad(\D,\Hil)}
\newcommand{\LDHpi}{{\mathcal L}\ad(\D_{\scriptscriptstyle\pi},\Hil_{\scriptscriptstyle\pi})}
\newcommand{\LDHpicl}{{\mathcal L}\ad(\widetilde{\D}_{\scriptscriptstyle\pi},\Hil_{\scriptscriptstyle\pi})}
\newcommand{\LDHr}{{\mathcal L}\ad(\lambda_\omega(\Ao),\H_\omega)}
\newcommand{\LpD}{{\mathcal L}\ad(\D)}
\newcommand{\LpDr}{{\mathcal L}\ad(\D_{\scriptscriptstyle\pi})}

\newcommand{\LAoH}{{\mathcal L}\ad(\Ao,\Hil)}
\newcommand{\QA}{{\mathcal Q}_{\Ao}(\A)}

\newcommand{\SSA}{{\mathcal S}_{\Ao}(\A)}

\newcommand{\betheo}{\begin{thm}}
\newcommand{\entheo}{\end{thm}}

\newcommand{\becor}{\begin{coroll}}
\newcommand{\encor}{\end{coroll}}

\newcommand{\belem}{\begin{lemma}}
\newcommand{\enlem}{\end{lemma}}

\newcommand{\beprop}{\begin{prop}}
\newcommand{\enprop}{\end{prop}}
\newcommand{\berem}{\begin{rem}$\!\!${\bf }$\;$\rm }
\newcommand{\beex}{\begin{example}$\!\!${\bf }$\;$\rm }
\newcommand{\enex}{ \end{example}}
\newcommand{\enrem}{ \end{rem}}

\newcommand{\rep}{{\mc R}(\A,\Ao)}
\newcommand{\crep}{{\mc R}_c(\A,\Ao)}
\newcommand{\curep}{{\mc R}_c^u(\A,\Ao)}

\newcommand{\hcrep}{{\mc R}_c(\H,\Ao)}
\newcommand{\hcurep}{{\mc R}_c^u(\H,\Ao)}

\def\H{{\mathcal H}}
\newcommand{\wmult}{\mbox{\raisebox{1pt}{$\scriptscriptstyle{
\square}$}}}

\newcommand{\omom}{\Omega^\omega}
\begin{document}
\title[Representable and continuous functionals]{Representable and continuous functionals on Banach quasi *-algebras}
\author{Maria Stella Adamo}
\author{Camillo Trapani}

\address{\textsc{Maria Stella Adamo}, Dipartimento di Matematica e Informatica, Universit\`a di Palermo, I-90123 Palermo, Italy}  \email{mariastella.adamo@community.unipa.it; msadamo@unict.it}
\address{\textsc{Camillo Trapani}, Dipartimento di Matematica e Informatica, Universit\`a di Palermo, I-90123 Palermo, Italy} \email{camillo.trapani@unipa.it}
\maketitle

\textsc{\abstractname{}}: In the study of locally convex quasi *-algebras an important role is played by representable linear functionals; i.e., functionals which allow a GNS-construction. This paper is mainly devoted to the study of the continuity of representable functionals in  Banach and Hilbert quasi *-algebras. Some other concepts related to representable functionals (full-representability, *-semisimplicity, etc) are revisited in these special cases. In particular, in the case of Hilbert quasi *-algebras, which are shown to be fully representable, the existence of a 1-1 correspondence between positive, bounded elements (defined in an appropriate way) and continuous representable functionals is proved.

\section{Introduction}\label{Sec1}
The main subject of this paper is the study of a class of linear functionals on a normed quasi *-algebra $(\A, \Ao)$ called {\em representable} functionals \cite{Trap1, bit_repr}. For short, they are exactly the linear functionals for which a GNS-like construction is possible.
As well as for the case of Banach *-algebras,
a certain amount of information on the structure of a Banach quasi
*-algebra can be obtained from the knowledge of the properties of
the family of its *-representations.
It is peculiar of this framework the need of enlarging the family of representations under consideration: in fact, in the theory of Banach *-algebras, representations take usually values in the C*-algebra $\mathcal{B}(\Hil)$  of bounded operators in some  Hilbert space $\H$; in the case of a locally convex quasi *-algebra   $(\A,\Ao)$, also unbounded representations are called on the stage.
A
*-representation of $(\A, \Ao)$ is, in fact, a *-homomorphism $\pi$ of $(\A,
\Ao)$ into the {\it partial} *-algebra $\LDH$ of closable linear
operators defined on a dense domain $\D$ of Hilbert space $\Hil$.  A
*-representation of a normed or Banach quasi *-algebra may be {\it
unbounded} in the sense that for some $a \in \A$, $\pi(a)$ can be
a true unbounded operator in $\H$ \cite{Ant2}.

From an historical point of view, the main motivation for theses studies comes from applications. In fact, quasi *algebras were introduced in the early 1980's by G. Lassner \cite{Lass2,Lass3}, for dealing with certain quantum physical models which do not fit into the celebrated algebraic approach of Haag and Kastler \cite{Haag} to quantum statistical physics and quantum field theory; see also \cite{Ant2, Bag6,Trap3}.

However since then many aspects of the theory have been investigated from a pure mathematical point of view (see, for instance, \cite{Ant2, Bag1, Bag5,bell_ct, Frag1,Frag2, Triolo, Trap1}) and several new notions have been introduced (e.g., full representability, full closability and *-semisimplicity).

In this paper, we examine some of these concepts and the interplay between them in the case of a Banach quasi *-algebra. We devote a particular attention to Hilbert quasi *-algebras which arise as completions of Hilbert algebras under the norm defined by their own inner product.

Going into details, in Section \ref{Sec2} we recall some definitions, properties and preliminary results needed for the work we are going to develop.  Section \ref{Sec3} is mainly devoted to examine the question as to whether a representable functional on a Banach quasi *-algebra $(\A,\Ao)$ is continuous. In particular, we prove that every continuous and representable functional on a Banach quasi *-algebra is {\em uniformly representable} which, roughly speaking, means that they satisfy a sort of Hilbert boundedness. Then, we study the sesquilinear forms associated to a representable and continuous functional  and we prove that they are necessarily everywhere defined and bounded. Thus in this case, the condition of full representability \cite{Frag2}  depends only on the {\em sufficiency} of this set of functionals, in the sense that this family of functionals separates points of $(\A,\Ao)$. In \cite{Bag2} it was shown that for a *-semisimple Banach quasi  *-algebra  this set of functionals  is sufficient. Thus a *-semisimple Banach quasi *-algebra is fully representable and the converse is true under an additional \textit{condition of positivity} (called $(P)$ in the text), which is satisfied in many interesting examples.

In Section \ref{Sec4} we focus on Hilbert quasi $^{\ast}-$algebras (which are proved to be \textit{always} fully representable) and, in particular, we characterize representable and continuous functionals on them. These functionals are in 1-1 correspondence with  {\em bounded} and {\em positive} elements of the Hilbert quasi *-algebra (Proposition \ref{defn_bounded} and Definition \ref{pos}).
We show that our definition is a generalization of the notion given in \cite{Frag2} and we study conditions under which these two notions coincide and, again, the condition of positivity (P) plays a role for their equivalence.  Moreover, if the Hilbert quasi *-algebra has a unit, the set of all bounded elements is a C*-algebra with respect to a {\em weak} multiplication on the Hilbert quasi *-algebra. This weak multiplication makes the Hilbert quasi *-algebra under consideration into a partial O*-algebra \cite{Ant2}.

Finally, in Section \ref{Sec5} we discuss the problem of continuity of representable functionals in the case where $\A$ is the Hilbert space of square integrable functions.

\section{Preliminaries and basic results}\label{Sec2}

\bedefi A {\em partial $^{\ast}-$algebra} is a $\mathbb{C}-$vector space $\A$ equipped of an involution $x\mapsto x^{\ast}$,  $x\in\A$, and a subset $\Gamma\subset\A\times\A$ with the following properties:
\begin{itemize}
\item[(i)] $(x,y)\in\Gamma$ if, and only if, $(y^{\ast},x^{\ast})\in\Gamma$;
\item[(ii)] if $(x,y)\in\Gamma$ and $(x,z)\in\Gamma$, then $(x,\lambda y+\mu z)\in\Gamma$ for every $\lambda,\mu\in\mathbb{C}$;
\item[(iii)] for every $(x,y)\in\Gamma$ there exists an element $x\cdot y\in\A$ with following the properties:
\begin{enumerate}
\item $x\cdot(y+\lambda z)=x\cdot y+\lambda(x\cdot z)$ and
\item $(x\cdot y)^{\ast}=y^{\ast}\cdot x^{\ast}$
\end{enumerate}
for every $(x,y),(x,z)\in\Gamma$ and $\lambda\in\mathbb{C}$.
\end{itemize}
\findefi

\beex\label{LDH}
Let $\Hil$ be a Hilbert space with inner product $\ip{\cdot}{\cdot}$ and let $\D$ be a dense linear subspace of $\Hil$. We denote by $\LDH$ the set of all closable operators $X$ in $\Hil$ such that the domain of $X$ is $\D$ and the domain of the adjoint $X^{\ast}$ contains $\D$, i.e.
$$\LDH=\left\{X:\D\to\Hil:\mathcal{D}(X)=\D, \mathcal{D}(X^{\ast})\supset\D\right\}.$$
$\LDH$ is a $\mathbb{C}-$vector space with the usual sum $X+Y$ and scalar multiplication $\lambda X$ for every $X,Y\in\LDH$ and $\lambda\in\mathbb{C}$. If we define the following involution $\dagger$ and partial multiplication $\wmult$
$$X\mapsto X^{\dagger}\equiv X^{\ast}\upharpoonright_{\D}\quad\text{and}\quad X\wmult Y=X^{\dagger\ast}Y$$
defined whenever $X$ is a left multiplier of $Y$ (or $Y$ a right multiplier of $X$), i.e. $Y\D\subset\mathcal{D}(X^{\dagger\ast})$ and $X^{\dagger}\D\subset\mathcal{D}(Y^{\ast})$, then $\LDH$ is a partial *-algebra.
In $\LDH$ several topologies can be introduced (see, \cite{Ant2}). Here, we will use only the {\em weak- topology} ${\sf t_w}$ defined by the family of seminorms
$$ p_{\xi,\eta} (X)=\ip{X\xi}{\eta},\quad \xi, \eta \in \D.$$

 A {\em partial O*-algebra} is a $^{\dagger}-$invariant subspace $\mathfrak{M}$ of $\LDH$ such that $X\wmult Y\in\mathfrak{M}$ whenever $X,Y\in\mathfrak{M}$ and $X$ is a left multiplier of $Y$.

{We denote by $\LpD$ the subset of elements $X$ of $\LDH$ such that $X\D\subseteq \D$ and $X^\dagger D\subset \D$. Then $\LpD$ is a *-algebra with respect to the the weak multiplication  $\wmult$ and one has, in particular
 $(X\wmult Y)\xi=XY\xi$, for every $\xi \in \D$.

 If $\M$ is a $^\dagger$-invariant family of operators of $\LDH$ the {\em weak commutant} $(\M, \D)'_w$ of $\M$ is defined as follows
 $$(\M, \D)'_w=\{ B\in {\mc B}(\H): \ip{BX\xi}{\eta}= \ip{B\xi}{X^\dag \eta}, \, \forall X \in \M, \xi, \eta \in \D\}.$$
 $(\M, \D)'_w$ is stable under involution and it is weakly closed, but it is not an algebra, in general.}
\enex

\bedefi A {\em quasi *--algebra} $(\A, \Ao)$ is a pair consisting of a vector space $\A$ and a *--algebra $\Ao$ contained in $\A$ as a subspace and such that
\begin{itemize}
\item[(i)] $\A$ carries an involution $a\mapsto a^*$ extending the involution of $\Ao$;
\item[(ii)] $\A$ is  a bimodule over $\A_0$ and the module multiplications extend the multiplication of $\Ao$. In particular, the following associative laws hold:
\begin{equation}\notag \label{eq_associativity}
(xa)y = x(ay); \ \ a(xy)= (ax)y, \; \forall \, a \in \A, \,  x,y \in \Ao;
\end{equation}
\item[(iii)] $(ax)^*=x^*a^*$, for every $a \in \A$ and $x \in \Ao$.
\end{itemize}
\findefi

\berem Every quasi *-algebra $(\A,\Ao)$ is a partial *-algebra: $(x,y)\in\Gamma$ if, and only if, $x\in\Ao$ or $y\in\Ao$.\enrem

A quasi *-algebra $(\A, \Ao)$ is
\emph{unital} if there is an element $\id \in \Ao$, such that
$a\id=a=\id a$, for all $a \in \A$; $\id$ is unique and called the
\emph{unit} of $(\A, \Ao)$.

\bedefi A {\em *-representation} of a quasi *-algebra $(\A,\Ao)$ is a *-homomorphism $\pi:\A\to\LDHpi$, where $\D_{\scriptscriptstyle\pi}$ is a dense subspace of the Hilbert space $\Hil_{\scriptscriptstyle\pi}$, with the following properties:
\begin{itemize}
\item[(i)] $\pi(a^{\ast})=\pi(a)^{\dagger}$ for all $a\in\A$;
\item[(ii)] if $a\in\A$ and $x\in\Ao$, then $\pi(a)$ is a left multiplier of $\pi(x)$ and $\pi(a)\wmult\pi(x)=\pi(ax)$.
\end{itemize}
{A *-representation $\pi$ of $(\A,\Ao)$ is called a {\em qu*-representation} if $\pi(\Ao)\subset \LpDr$.}

A *-representation $\pi$ is {\em faithful} if $a\neq0$ implies $\pi(a)\neq0$. $\pi$ is {\em cyclic} if $\pi(\Ao)\xi$ is dense in $\Hil_{\scriptscriptstyle\pi}$ for some $\xi\in\D_{\scriptscriptstyle\pi}$.
If $(\A,\Ao)$ has a unit $e$, we suppose that $\pi(e)=I_\D$, the identity operator of $\D$.
\findefi

 The {\em closure $\widetilde{\pi}$} of a *-representation $\pi$ of the quasi *-algebra $(\A,\Ao)$ in $\LDHpi$ is defined as
$$\widetilde{\pi}:\A\to\LDHpicl,\quad a\mapsto\overline{\pi(a)}_{\upharpoonright\widetilde{\D}_{\scriptscriptstyle\pi}}$$
where $\widetilde{\D}_{\scriptscriptstyle\pi}$ is the completion of $\D_{\scriptscriptstyle\pi}$ with respect to the graph topology, i.e. the topology defined by the seminorms $\eta\in\D_{\scriptscriptstyle\pi}\mapsto\left\|\pi(a)\eta\right\|$ for every $a\in\A$.
{A *-representation $\pi$ is said to be {\em closed} if $\pi=\widetilde{\pi}$.}

\bedefi
Let $(\A, \Ao)$ be a quasi *-algebra.  A linear functional $\omega$ on a $\A$ is said to be {\em representable} if it satisfies the following
conditions:
\begin{itemize}
    \item[(L.1)]$\omega(x^*x) \geq 0, \quad \forall x \in \Ao$;
    \item[(L.2)]$\omega(y^*a^* x)= \overline{\omega(x^*ay)}, \quad
\forall x,y \in \Ao, \forall a \in \A$;
    \item[(L.3)]$\forall a \in \A$, there exists $\gamma_a >0$ such
    that $$|\omega(a^*x)| \leq \gamma_a \omega(x^*x)^{1/2}, \quad \forall x
    \in \Ao.$$\end{itemize}
 The family of representable functionals is denoted by $\rep$.
\findefi
\berem  The following statements are easily proved. \begin{itemize} \item[(a)]If $\omega_1,\omega_2\in \rep$, then $\omega_1 + \omega_2 \in \rep$ as well as $\lambda\omega_1\in \rep$, for every $\lambda\geq0$.

\item[(b)] If $\omega \in \rep$ and $z \in \Ao$, then the linear functional $\omega_z$, defined by
$$ \omega_z (a):=\omega(z^* a z), \quad a \in \A,$$
is representable.

{\item[(c)] If $(\A, \Ao)$ has a unit $e$ and $\omega \in \rep$ is such that $\omega(x)=0$, for every $x\in \Ao$, then $\omega\equiv 0$.}
\end{itemize}
\enrem

\medskip
\betheo\cite[Theorem 3.5]{Trap1}\label{thtp}  Let $(\A, \Ao)$ be a quasi *-algebra with unit
$e$ and let $\omega$ be a representable linear functional on $(\A, \Ao)$. Then, there
exists a (closed) cyclic *-representation ${\pi}_\omega$ of $(\A,\Ao)$, with cyclic vector $\xi_\omega$,
such that $$ \omega(a)=\ip{{\pi}_\omega(a)\xi_\omega}{\xi_\omega}, \quad \forall a \in \A.$$ This
representation is unique up to unitary equivalence.\entheo

\medskip
As in \cite[Definition 2.1]{Trap1}, given a quasi *-algebra $(\A,\Ao)$, we denote by $\QA$ the set of all sesquilinear forms on $\A\times\A$ such that
\begin{itemize}
\item[(i)] $\varphi(a,a)\geq0$ for every $a\in\A$.
\item[(ii)] $\varphi(ax,y)=\varphi(x,a^{\ast}y)$ for every $a\in\A$ and $x,y\in\Ao$
\end{itemize}

\medskip Finally, we recall that a Hilbert seminorm on a vector space $V$ is a seminorm $p$ satisfying
$$ p(a+b)^2+ p(a-b)^2 = 2 p(a)^2 +2 p(b)^2,\quad \forall a,b \in V.$$

With these definitions we can characterize representable functionals as follows.
\begin{prop} \label{thm_GNS}Let $(\A, \Ao)$ be a quasi *-algebra with unit $e$ and $\omega$ a linear functional on $\A$ satisfying (L.1) and (L.2). The following statements are equivalent.
\begin{itemize}
\item[(i)] $\omega$ is representable.
\item[(ii)] There exist a *-representation $\pi$ defined on a dense domain $\D_\pi$ of a Hilbert space $\H_\pi$ and a vector $\zeta \in \D_\pi$ such that
$$ \omega(a) = \ip{\pi(a)\zeta}{\zeta}, \quad \forall a \in \A.$$
\item[(iii)] There exists a sesquilinear form $\Omega^\omega \in \QA$ such that $$\omega(a)=\Omega^\omega (a,e), \quad \forall a \in \A.$$
\item[(iv)] There exists a Hilbert seminorm $p$ on $\A$ such that
$$|\omega(a^*x)| \leq p(a) \omega(x^*x)^{1/2}, \quad \forall x
    \in \Ao.$$
    \end{itemize}
\end{prop}
\begin{proof} (i) implies (ii) by Theorem \ref{thtp}. Suppose now that (ii) holds and define:
$$\Omega^\omega (a,b)= \ip{\pi(a)\zeta}{\pi(b)\zeta},\quad a,b \in \A.$$
It is then easy to check that $\Omega^\omega$ has the desired properties. This proves (iii).

Now suppose that (iii) holds and define $p(a)=\Omega^\omega(a,a)^{1/2}$. Then (iv) follows immediately from the Cauchy-Schwarz inequality.

Finally suppose that (iv) holds. 
Then, clearly
$$|\omega(a^{\ast}x)|\leq p(a)\omega(x^{\ast}x)^{\frac12}\leq \gamma_a\omega(x^{\ast}x)^{1/2}$$
where, for instance, $\gamma_a:=\left(1+p(a)^2\right)^{1/2}$. Hence, $\omega$ is representable.
\end{proof}

\berem \label{rem_L3rev} Once the GNS representation $\pi_\omega$ associated to $\omega\in \rep$ we can write the condition in (iv) of Theorem \ref{thm_GNS} as follows
$$|\omega(a^*x)| \leq (1+ \|\pi_\omega(a)\xi_\omega\|^2)^{1/2} \omega(x^*x)^{1/2}, \quad \forall a\in \A,\, x
    \in \Ao.$$
\enrem

Let now $(\A,\Ao)$ be a quasi *-algebra with unit $e$ and $\omega\in \rep$. Let $N_\omega=\{x\in \Ao:\, \omega(x^*x)=0\}$ and $\H_\omega$ the Hilbert space completion of $\Ao/N_\omega$ with respect to the inner product
$$ \ip{\lambda_\omega(x)}{\lambda_\omega(x)}=\omega(y^*x), \;x,y \in \Ao,$$
where $\lambda_\omega(x)=x+N_\omega$. $\H_\omega$ is nothing but the carrier space of the GNS representation constructed in \cite{Trap1}.
For every $a\in \A$, the linear functional $F_a$ on $\lambda_\omega(\Ao)$ defined by
$$F_a(\lambda_\omega(x))=\omega(a^*x), \quad x \in \Ao$$
is well-defined and bounded by (L.3). Hence there exists a unique $\xi(a) \in \H_\omega$ such that
$$ F_a(\lambda_\omega(x))=\omega(a^*x)=\ip{\lambda_\omega(x)}{\xi(a)}, \quad \forall x\in \Ao.$$
We define a linear map $T_\omega:\A\to \H_\omega$ by $T_\omega a=\xi(a)$, $a \in \A$. Then we have
$$ \omega(a)= \ip{T_\omega a}{\xi_\omega}, \quad \forall a \in \A,$$
where $\xi_\omega=\lambda_\omega(e)$ is the cyclic vector of the GNS representation associated to $\omega$.
Then, we can state the following
\begin{prop}\label{prop_operatorT}
Let $(\A,\Ao)$ be a quasi *-algebra with unit $e$ and $\omega\in \rep$. Then there exists a linear operator $T_\omega:\A \to \H_\omega$ such that
$$ \omega(a)= \ip{T_\omega a}{\xi_\omega}, \quad \forall a \in \A,$$
where $\xi_\omega$ is the cyclic vector of the GNS representation associated to $\omega$.
\end{prop}
To every  $\omega \in \rep$ we can associate two sesquilinear forms which are useful for our discussion. The first one $\Omega^\omega$, already introduced in (iii) of Proposition \ref{thm_GNS}, can be defined through the GNS representation $\pi_\omega$, with cyclic vector $\xi_\omega$.   In fact, we put
\begin{equation}\label{eq_omom}\Omega^\omega (a,b)=\ip{\pi_\omega(a)\xi_\omega}{\pi_\omega(b)\xi_\omega}, \quad a,b \in \A.\end{equation}
As we have seen $\Omega^\omega\in \QA$ and
$ \omega(a)= \Omega^\omega (a,e),$ for every $a \in \A.$

The second sesquilinear form, which we denote by $\vp_\omega$, is defined only on $\Ao\times\Ao$ by
\begin{equation} \label{sesqu_ass} \vp_\omega(x,y)= \omega(y^*x), \quad x,y \in \Ao.\end{equation}
It is clear that $\Omega^\omega$ extends $\vp_\omega.$
It is easy to see that
\begin{itemize}
\item[(i)] $\vp_\omega(x,x) \geq 0,$ for every $x \in \Ao$.
\item[(ii)] $\vp_\omega(xy,z)=\vp_\omega (y,x^*z)$ for every $x, y, z \in \Ao$.
\end{itemize}

We define a partial order in $\rep$ as follows. If $\omega, \theta \in \rep$ we say that $\omega\leq \theta$ if $\Omega^\omega (a,a)\leq \Omega^\theta (a,a)$, for every $a\in \A$.

{By a slight modification of standard arguments (see, e.g., \cite[Proposition 9.2.3]{Ant2}), one can prove the following

\begin{prop}\label{lemma_214} Let $\omega, \theta \in \rep$ and let $\pi_\omega$ denote the GNS-representation associated to $\omega$. If $\theta \leq \omega$, there exists an operator $S\in (\pi_\omega(\A), \D_\pi)'_w$, with $0\leq S\leq I$, such that
$$\Omega^\theta(a,b)= \ip{\pi_\omega(a)\xi_\omega}{S\pi_\omega(b)\xi_\omega}, \quad \forall a,b \in \A.$$

\end{prop}}

For future use we give the following

\begin{lemma}\label{lemma_dom} Let $\omega, \theta \in \rep$ with $\omega\leq \theta$. Then $\theta-\omega \in \rep.$
\end{lemma}
\begin{proof} The conditions (L1) and (L2) are obviously satisfied. We check (L3). For every $a\in \A$ and $x\in \Ao$, using the Cauchy-Schwarz inequality for the positive sesquilinear form $\Omega^\theta -\Omega^\omega$, we get

\begin{align*}|(\theta-\omega)(a^*x)|&=|(\Omega^\theta -\Omega^\omega)(a^*x,e)| = |(\Omega^\theta -\Omega^\omega)(x,a)| \\ &\leq (\Omega^\theta -\Omega^\omega)(x,x)^{1/2}(\Omega^\theta -\Omega^\omega)(a,a)^{1/2}\\ &=
(\theta-\omega)(x^*x)^{1/2}(\Omega^\theta -\Omega^\omega)(a,a)^{1/2}
\end{align*}
\end{proof}

\bedefi A  quasi *-algebra $(\A,\Ao)$ is called a {\em normed quasi *-algebra} if a norm
$\|\cdot\|$ is defined on $\A$ with the properties
\begin{itemize}

\item[(i)]$\|a^*\|=\|a\|, \quad \forall a \in \A$;
\item[(ii)] $\Ao$ is dense in $\A$;
\item[(iii)]for every $x \in \Ao$, the map $R_x: a \in \A \to ax \in \A$ is continuous in
$\A$.
\end{itemize}
If $(\A,\| \cdot \|) $ is a Banach space, we say that $(\A,\Ao)$ is a {\em Banach quasi *-algebra}.\label{def} The norm topology of $\A$ will be denoted by $\tau_n$.
\findefi The continuity of the involution implies that
\begin{itemize}
\item[(iii')]for every $x \in \Ao$, the map $L_x: a \in \A \to xa \in \A$ is continuous in
$\A$.
\end{itemize}
If $x \in \Ao$, we put
$$\|x\|_0:= \max\left\{\sup_{\|a\|\leq 1}\|ax\|, \sup_{\|a\|\leq 1}\|xa\|\right\}.$$
Then,
\begin{align*}
&\|ax\|\leq \|a\|\|x\|_0, \quad \forall a \in \A, \, x \in \Ao;\\
&\|x^*\|_0 = \|x\|_0, \quad \forall  x \in \Ao.
\end{align*}

A Banach quasi *-algebra $(\A, \Ao)$ is called a {\em proper CQ*-algebra} if $\Ao[\|\cdot\|_0]$ is a C*-algebra.

\beex Easy examples of Banach quasi *-algebras are provided by $L^p$-spaces both in the commutative case and in the noncommutative one; see, e.g., Example \ref{ex_ellepi} and Section \ref{Sec5} below, and \cite{btt_meas}.
\enex

If $(\A, \Ao)$ is a normed quasi *-algebra, we denote by $\crep$ the subset of $\rep$ consisting of continuous functionals.
As shown in \cite{Frag2}, if $\omega\in \crep$, then the sesquilinear form $\vp_\omega$ defined in \eqref{sesqu_ass}, is {\em closable}; that is, $\vp_\omega(x_n,x_n)\to 0$, for every sequence $\{x_n\}\subset \Ao$ such that $$\mbox{$\|x_n\|\to 0$ and $\vp_\omega(x_n-x_m, x_n-x_m)\to 0$}.$$

In this case, $\vp_\omega$  has a closed extension $\overline{\vp}_\omega$ to a dense domain $D(\overline{\vp}_\omega)$ containing $\Ao$. Thus, a natural question arises: under which conditions one gets the equality $D(\overline{\vp}_\omega)=\A$? An answer to this question will be given in Proposition \ref{prop1}.

\medskip

Consider now the set
$$\A_{\mc R}:= \bigcap_{\omega \in {\mc R}_c(\A,\Ao)}D(\overline{\vp_\omega}).$$

If ${\mc R}_c(\A,\Ao)=\{0\}$, we put $\A_{\mc R}=\A$. Note that, if for every $\omega \in {\mc R}_c(\A,\Ao)$, $\vp_\omega$ is jointly
continuous with respect to the topology $\tau_n$ defined by the norm $\|\cdot\|$, we get
$\A_{\mc R}=\A$.

\medskip Set
$$\Ao^+:=\left\{\sum_{k=1}^n x_k^* x_k, \, x_k \in \Ao,\, n \in {\mb N}\right\}.$$
Then $\Ao^+$ is a wedge in $\Ao$ and we call the elements of $\Ao^+$ \emph{positive elements of} $\Ao$.
As in \cite{Frag2}, we call \emph{positive elements of} $\A$ the members of $\overline{\Ao^+}^{\tau_n}$. We set
$\A^+:=\overline{\Ao^+}^{\tau_n}$.

\bedefi A family of positive linear functionals $\mc F$ on
$(\A[\tau_n], \Ao)$ is called {\em sufficient} if for every $a \in
\A^+$, $a \neq 0$, there exists $\omega \in {\mc F}$ such that $\omega
(a)>0$. \findefi

 \bedefi\label{fully_rep} A normed quasi $^{\ast}$-algebra
$(\A[\tau_n], \Ao)$ is called {\em fully representable} if ${\mc
R}_c(\A,\Ao)$ is sufficient and $\A_{\mc R}=\A$.
\findefi

 \berem \label{rem_217} If $(\A,\Ao)$ has a unit $e$, the condition of sufficiency required in Definition \ref{fully_rep} joined with the following condition of positivity
\begin{equation}\tag{P} a\in\A\;\text{and}\;\omega_x(a)\geq0\;\text{for every}\;\omega\in\mathcal{R}_c(\A,\Ao)\;\text{and}\;x\in\Ao\;\;\Rightarrow\;\;a\geq0
\end{equation}
tells us that $\omega(a)=0$ for every $\omega\in\crep$ implies $a=0$.
\enrem

We denote by $\SSA$ the subset of $\QA$ consisting of all continuous sesquilinear forms $\Omega:\A\times\A\to\mathbb{C}$ such that
$$ |\Omega(a,b)|\leq \|a\|\|b\|, \quad \forall a, b\in \A.$$
By defining $$\|\Omega\|=\displaystyle\sup_{\|a\|=\|b\|=1}\left|\Omega(a,b)\right|,$$ one obviously has $\|\Omega\| \leq 1$, for every $\Omega \in \SSA$.
\bedefi\label{def1}
A normed quasi *-algebra $(\mathfrak{A}[\tau_n],\Ao)$ is called {\em *-semi\-simple} if, for every $0\neq a\in\A$, there exists $\Omega\in\mathcal{S}_{\Ao}(\A)$ such that $\Omega(a,a)>0$.
\findefi

\bedefi \label{defn_219}Let $(\A[\tau_n],\Ao)$ be a normed quasi *-algebra. Let $a\in\A$ and consider the left regular representation $L_a:\Ao\to\A$ defined as $L_a(x)=ax$. If $L_a$ is closable, for every $a\in\A$, then the normed quasi *-algebra $(\A[\tau_n],\Ao)$ is said to be {\em fully closable}.
\findefi

\berem The left regular representation in general is not continuous on $\Ao[\tau_n]$. The right regular representation $R_a:\Ao\to\A$ with $R_a(x)=xa$ is closable, for every $a \in \A$ if, and only if, $L_a$ is closable, for every $a \in \A$.
\enrem

\section{Continuity and full representability}\label{Sec3}
In this section we will investigate on the continuity of representable linear functionals on $(\A, \Ao)$.

{\begin{thm} Let $(\A, \Ao)$ be a Banach quasi *-algebra. The following statements are equivalent.
\begin{itemize}
\item[(i)] Every $\omega \in \rep$ is bounded; i.e., $\rep=\crep$.
\item[(ii)]Every *-representation $\pi$ of $(\A, \Ao)$ is weakly continuous from $\A[\|\cdot\|]$ into $\LDHpi[\tau_w]$.
\item[(iii)] For every $\omega\in \rep$, $\omega \neq 0$, there exists a nonzero $\theta \in \crep$ such that $\theta \leq \omega$.
\end{itemize}

\end{thm}
\begin{proof} (i) $\Rightarrow$ (ii): Let $\pi$ be a *-representation of $(\A, \Ao)$. Then, for every $\xi\in \D_{\scriptscriptstyle \pi}$ the linear functional $\omega(a)= \ip{\pi(a)\xi}{\xi}$ is representable and, therefore bounded. This easily implies that $\pi: \A[\|\cdot\|] \to \LDHpi[\tau_w]$ is continuous.

(ii) $\Rightarrow$ (iii): Let $\omega\in \rep$ and $\pi_{\omega}$ the corresponding GNS-representation (which is weakly-continuous by assumption) with cyclic vector $\xi_\omega$.
Then for every $\xi, \eta\in \D_\omega$ there exists $\gamma_{\xi,\eta}>0$ such that
$$ | \ip{\pi_\omega(a)\xi}{\eta}| \leq \gamma_{\xi,\eta} \|a\|, \quad \forall a \in \A.$$
In particular, for the cyclic vector $\xi_\omega$, we have
$$ |\omega(a)| =| \ip{\pi_\omega(a)\xi_\omega}{\xi_\omega}|\leq \gamma_{\xi_\omega,\xi_\omega} \|a\|, \quad \forall a \in \A.$$
Then (iii) holds with the obvious choice of $\theta=\omega$.

(iii) $\Rightarrow$ (i):
By the assumption, the set ${\mc K}_\omega=\{ \theta \in \hcrep: \theta \leq \omega\}$ is a nonempty partially ordered (by $\leq$) set. Let ${\mc W}$ be a totally ordered subset of ${\mc K}_\omega$.
Then $$\lim_{\theta \in {\mc W}}\theta(a)$$ exists for every $a \in \A$. Indeed, the set of numbers $\{\Omega^\theta (a,a); \theta \in {\mc W}\}$ is increasing and bounded from above by  $\omom(a,a)$.
We set, for every $a\in \H$,  $$\Lambda (a,a)= \lim_{\theta \in {\mc W}}\Omega^\theta (a,a).$$ Then $\Lambda$ satisfies the equality
$$ \Lambda(a+b, a+b)+\Lambda (a-b, a-b)=2 \Lambda(a,a) +2 \Lambda (b,b), \quad \forall a,b \in \A,$$  hence we can then define $\Lambda$ on $\A\times\A$ using the polarization identity. It is readily checked that $\Lambda \in \QA.$

Now we put $\omega^\circ (a)= \Lambda (a,e)$. Then,
$ \omega^\circ (a) = \lim_{\theta \in {\mc W}}\theta(a).$   We now prove that $\omega^\circ \in \rep$. It is clear that $\omega^\circ$ is a linear functional on $\A$ and $\omega^\circ\leq \omega$. The conditions (L1) and (L2) are obviously satisfied. We prove (L3).
Let $a\in \A$ and $x \in \Ao$. Then, taking into account Remark \ref{rem_L3rev},
\begin{align*} |\omega^\circ(a^*x)|&= \lim_{\theta \in {\mc W}}|\theta(a^*x)|\leq { \lim_{\theta \in {\mc W}}}(1 +\Omega^\theta(a,a))^{1/2}\lim_{\theta \in {\mc W}}\theta(x^*x)^{1/2} \\ & \leq (1 +{\Lambda(a,a)})^{1/2}\omega^\circ(x^*x)^{1/2} .\end{align*}
We show now that $\omega^\circ$ is bounded. For every $a \in \H$ the set $\{|\theta (a)|; \theta \in {\mc W}\}$ is bounded; indeed, for every $\theta \in {\mc W}$, we get
$$ |\theta (a)|=|\Omega^\theta(a,e)| \leq\Omega^\theta(a,a)^{1/2}\Omega^\theta (e,e)^{1/2} \leq \omom (a,a)^{1/2} \omom (e,e)^{1/2}.$$
By the uniform boundedness principle, we conclude that there exists $\gamma>0$ such that $|\theta (a)|\leq \gamma \|a\|$, for every $\theta \in {\mc W}$ and for every $a \in \H$.
Hence, $$|\omega^\circ (a)| = \lim_{\theta \in {\mc W}}|\theta(a)| \leq \gamma \|a\|, \quad \forall a \in \H.$$
Then, ${\mc W}$ has an upper bound. Then, by Zorn's lemma, ${\mc K}_\omega$ has a maximal element $\omega^{\scriptscriptstyle \bullet}$. It remains to prove that $\omega=\omega^{\scriptscriptstyle \bullet}$. Assume, on the contrary that $\omega > \omega^{\scriptscriptstyle \bullet}$.
Let us consider the functional $\omega-\omega^{\scriptscriptstyle \bullet}$, which is nonzero and representable by Lemma \ref{lemma_dom}. Then, there exists $\sigma\in \crep$ such that $\omega-\omega^{\scriptscriptstyle \bullet}\geq \sigma$. Hence, $\omega\geq \omega^{\scriptscriptstyle \bullet} + \sigma$, contradicting the maximality of $\omega^{\scriptscriptstyle \bullet}$. Then $\omega=\omega^{\scriptscriptstyle \bullet}$ and, therefore, $\omega$ is continuous.
\end{proof}

\berem The equivalence of (i) and (ii) of the previous theorem holds also in the case when $(\A,\Ao)$ is only a normed quasi *-algebra. The proof of (iii) $\Rightarrow$ (i) is inspired by a well known result of the theory of Banach *-algebras \cite[Lemma 5.5.5]{dales}.
\enrem

\beex \label{ex_ellepi}{Let us  consider the Banach quasi *-algebra (actually, a proper CQ*-algebra)  $(L^p(I, d\lambda), C(I))$ where $I$ is a compact interval of the real line and $\lambda$ the Lebesgue measure. } For $1\leq p<2$ (see \cite{Bag4,Frag2}), the sole representable and continuous functional is the trivial one. Going over the threshold $p=2$, these Banach quasi *-algebras are fully representable, thus we have a so huge amount of representable and continuous functionals to separate their points.
\enex

In general, a continuous functional $\omega$ satisfying (L1) and (L2) need not satisfy (L3) and therefore it is not representable, as the next examples show.

\beex
Let us  consider the Banach quasi *-algebra $(L^2(I, d\lambda), C(I))$ ($I$ and $\lambda$ as in Example \ref{ex_ellepi}). Let $h\in L^2(I)$, $h\geq 0$. We define
$$ \omega(f)= \int_I f(x)h(x)d\lambda(x), \quad f \in L^2(I).$$
It is clear that $\omega$ satisfies (L.1) and (L.2).
Condition (L.3) requires that, for every $f\in L^2(I)$, there exists $\gamma_f>0$ such that
$$\left| \int_I \overline{f(x)} \phi(x) h(x)d\lambda(x) \right| \leq \gamma_f\left(\int_I |\phi(x)|^2 h(x) d\lambda(x)\right)^{1/2} , \quad \forall \phi\in C(I).$$
Since $h(x)>0$ a.e., this inequality implies that $f\in L^2(I, hd\lambda)$. Were $L^2(I, d\lambda)\subseteq L^2(I, hd\lambda)$ then we should have that $f\sqrt{h} \in L^2(I, d\lambda)$, for every $f \in L^2(I, d\lambda)$. This in turn implies that $h\in L^\infty(I, d\lambda)$.
So it suffices to pick $h\in L^2(I, d\lambda)\setminus L^\infty(I, d\lambda)$ to get the desired example.
\enex

\beex \label{ex-sesqforms} Let $\D$ be a dense domain  in a Hilbert
space $\H $ and $\|\cdot\|_1$ a norm on $\D$, stronger than the
Hilbert norm $\|\cdot \|$. Let $\textsf{B}(\D,\D)$ denote \index{$\textsf{B}(\D,\D)$} the
vector space of all {\em jointly} continuous sesquilinear forms on
$\D\times \D$, with respect to $\|\cdot\|_1$. The map $\vp \to \vp^*$, with $$ \vp^*(\xi,\eta)=
\overline{\vp(\eta,\xi)},$$ defines an involution in
$\textsf{B}(\D,\D)$ (see, \cite[Example 3.8]{Trap1}).

 We denote by $\mathfrak{L}^\dag(\D)$  the *--subalgebra
of ${\mathcal L}^\dag(\D)$ consisting of all operators $A\in {\mathcal L}^\dag(\D)$ such
that both $A$ and $A^\dag$ are continuous from $\D[\|\cdot\|_1]$
into itself.

 Every $A \in \mathfrak{L}^\dag(\D)$ defines a  sesquilinear
form $\vp_A\in \textsf{B}(\D,\D)$  by 
$$ \vp_A(\xi,\eta) =\ip{A\xi}{\eta}, \quad \xi, \eta\in \D.$$
$$\textsf{B}\ad(\D)=\{\vp_A : A\in \mathfrak{L}^\dag(\D)\}.$$  It is easily seen that
$\vp_A^*=\vp_{A\ad}$, for every $A \in  \mathfrak{L}^\dag(\D)$. \\ For $\vp
\in \textsf{B}(\D,\D)$, $\vp_A \in \textsf{B}\ad(\D)$, the multiplications are defined as
\begin{align*}
& (\vp\circ \vp_A)(\xi, \eta)= \vp(A\xi,\eta), \quad \xi,\eta \in
\D\\
&(\vp_A\circ \vp)(\xi, \eta)= \vp(\xi,A\ad\eta), \quad \xi,\eta \in \D.
\end{align*}
With these operations and involution, $(\textsf{B}(\D,\D),
\textsf{B}\ad(\D))$ is a quasi *-algebra.
A norm on $\textsf{B}(\D,\D)$ is defined by
$$ \|\vp\|:= \sup_{\|\xi\|_1=\|\eta\|_1=1}|\vp (\xi,\eta)|.$$
Then the pair $(\overline{\textsf{B}\ad(\D)},\textsf{B}\ad(\D))$, where $\overline{\textsf{B}\ad(\D)}$ denotes the $\|\cdot\|$-closure of $\textsf{B}\ad(\D)$, is a Banach quasi*-algebra.

 For every $\xi \in \D$, we define \index{$ \omega_\xi$}
$$ \omega_\xi (\vp) =\vp(\xi, \xi), \quad \vp \in \textsf{B}(\D,\D).$$
Then $\omega_\xi$ is a linear functional on $\textsf{B}(\D,\D)$. 
%
%

The functional $\omega_\xi$ is representable if, and only if, 
$\vp$ is bounded in the first variable on the subspace
${\mc M}_\xi= \{A\xi; A \in \mathfrak{L}^\dag(\D) \}$. 
 Hence,  every $\omega_\xi$ is continuous but it need not be representable. \enex

As we mentioned before, if $\omega\in \crep$, then the form $\vp_\omega$ defined in \eqref{sesqu_ass} is closable. For Banach quasi *-algebras this result can be improved.
\begin{prop}\label{prop1}
Let $(\A, \Ao)$ be a Banach quasi *-algebra with unit $e$,  $\omega \in \crep$ and $\vp_{\omega}$ the associated sesquilinear form on $\Ao \times \Ao$ defined as in \eqref{sesqu_ass}. Then $D(\overline{\vp}_\omega)=\A$;  hence $\overline{\vp}_\omega$ is everywhere defined and bounded.
\end{prop}
\begin{proof} Since $\omega$ is representable, there exists a Hilbert space $\H_\omega$, a linear map $\lambda_\omega: \Ao \to \H_\omega$ and a *-representation $\pi_\omega$ with values in $\LDHr$ such that
$$ \omega(y^*ax)= \ip{\pi_\omega(a)\lambda_\omega(x)}{\lambda_\omega(y)}, \quad \forall a \in \A, \, x,y \in \Ao. $$
Then, by the properties of the norm on $(\A,\Ao)$, for every $a \in \A$ and $x,y \in \Ao$,
\begin{equation}\label{eq_one} |\ip{\pi_\omega(a)\lambda_\omega(x)}{\lambda_\omega(y)}|=|\omega(y^*ax)|\leq \gamma \|a\|\|x\|_0\|y\|_0.  \end{equation}
Now, consider the sesquilinear form $\omom$ defined in \eqref{eq_omom}. As already noticed,
$\omom$ extends $\vp_{\omega}$. It remains to show that $\omom$ is closable.

Suppose now that $\{a_n\}$ is a sequence in $\A$ such that $\|a_n \|\to 0$ and $\omom(a_n-a_m,a_n-a_m)=\|\pi_\omega(a_n-a_m)\xi_\omega\|^2\to 0$.
Then the sequence $\pi_\omega(a_n)\xi_\omega$ converges to a vector $\zeta\in \H_\omega$. Thus,
$$\ip{\pi_\omega(a_n)\xi_\omega}{\lambda_\omega(y)}\to \ip{\zeta}{\lambda_\omega(y)}, \quad \forall y \in \Ao.$$
Using \eqref{eq_one}, we obtain $\zeta=0$. Hence $\omom (a_n,a_n)\to 0$; i.e., $\Omega^\omega$ is closable. Thus $\omom$ is closed and everywhere defined, hence bounded.
It is clear that $\omom=\overline{\vp}_\omega$.
\end{proof}}
{
In a normed quasi-algebra it appears natural to strengthen a bit the notion of representable functional as follows.
\bedefi Let $(\A,\Ao)$ be a normed quasi *-algebra. We say that $\omega\in \rep$ is {\em uniformly representable} if there exists $\gamma >0$ such that
\begin{equation}\label{eqn_L3unif}|\omega(a^*x)|\leq \gamma \|a\| \omega(x^*x)^{1/2}, \quad \forall a \in \A, x\in \Ao.\end{equation}
 The set of {uniformly representable} linear functionals is denoted by $\curep$.
\findefi

\begin{prop}\label{prop_charuniformly_rep} Let $(\A,\Ao)$ be a Banach quasi *-algebra, $\omega\in \rep$ and $T_\omega$ the operator defined in Proposition \ref{prop_operatorT}.
The following statements are equivalent.
\begin{itemize}
\item[(i)]$\omega \in \crep$; i.e., $\omega$ is bounded.
\item[(ii)] $T_\omega$ is bounded.
\item[(iii)]$\omega\in \curep$; i.e., $\omega$ is uniformly representable.
\end{itemize}
\end{prop}
\begin{proof}
(i)$\Rightarrow$(ii): Since $T_\omega$ is everywhere defined, it is enough to prove that $T_\omega$ is closable.
Let $\{a_n\}$ be a sequence in $\A$ such that $\|a_n\|\to 0$ and $T_\omega a_n \to \xi \in \H_\omega$. Then, for every $x\in \Ao$,
$$ \ip{\lambda_\omega(x)}{\xi}=\lim_{n\to \infty}\ip{\lambda_\omega(x)}{T_\omega a_n}=\lim_{n\to \infty}\omega(a_n^*x).$$
By the asumptions and the properties of the norm in $(\A,\Ao)$, we have
$$|\omega(a_n^*x)| \leq  c\|a_n^*x\|\leq c\|a_n\|\|x\|_0 \to 0.$$
for some $c>0$. Hence,
$$\ip{\lambda_\omega(x)}{\xi}=0, \quad \forall x \in \Ao.$$
The density of $\lambda_\omega(\Ao)$ in $\H_\omega$ implies that $\xi=0$, thus $\T_\omega$ is closable. The statement folows by the closed graph theorem.

(ii)$\Rightarrow$(iii): If $T_\omega$ is bounded, then
$$|\omega(a^*x)|=|\ip{\lambda_\omega(x)}{T_\omega a}|\leq \|\lambda_\omega(x)\| \|T_\omega a\|\leq \|T_\omega\|\|a\| \omega(x^*x)^{1/2}$$
for every $a \in \A$ and $x\in \Ao$, that is $\omega \in \curep$.

(iii)$\Rightarrow$(i): This is obvious. \end{proof}
}

Using the previous results, we conclude this section with giving some conditions for the full-represen\-tabi\-lity of a Banach quasi *-algebra.

{ \begin{thm}\label{thm_fullrep_semis} Let $(\A, \Ao)$ be a Banach quasi *-algebra with unit $e$. The following statements are equivalent. 
\begin{itemize}
\item[(i)]$\crep$ is sufficient.
\item[(ii)]$(\A,\Ao)$ is fully representable.
\end{itemize}
If the condition of positivity $(P)$ holds, (i) and  (ii) are equivalent to the following
\begin{itemize}
\item[(iii)]$(\A,\Ao)$ is *-semisimple.
\end{itemize}
\end{thm}}
\begin{proof} {The equivalence between (i) and (ii) follows from the very definitions and from Proposition \ref{prop1}.

Suppose now that the condition of positivity $(P)$ is valid. We will prove that (iii) is equivalent to (ii) $\sim$ (i).}

(ii) $\Rightarrow$ (iii): Let $\mathcal{S}_{\Ao}(\A)$ be as in Definition \ref{def1}. First, we notice that every $\Omega\in\mathcal{S}_{\Ao}(\A)$ can be written as $\overline{\vp}_{\omega}$, for some $\omega\in\mathcal{R}_c(\A,\Ao)$. Indeed, if we put
$$\omega_{\Omega}(a):=\Omega(a,e), \quad a\in\A,$$
then, it is easily seen that $\omega_{\Omega}$ is continuous and representable, so $\omega_{\Omega}\in\mathcal{R}_c(\A,\Ao)$
and $\overline{\vp}_{\omega_{\Omega}}=\Omega$.

On the other hand, consider a linear functional $0\neq\omega\in\crep$ and let $\overline{\vp}_{\omega}$ be the sesquilinear form associated to it as in \eqref{sesqu_ass}. By Proposition \ref{prop1} $D( \overline{\vp}_\omega)=\A$, thus $\overline{\vp}_\omega$ is bounded. If we put $\vp'_\omega= \overline{\vp}_\omega /\|\overline{\vp}_\omega\|$, then $\vp'_\omega\in \SSA.$

Let $a\in\A$ be such that $\Omega(a,a)=0$ for every $\Omega\in\mathcal{S}_{\Ao}(\A)$. For what we have just shown, it enough to prove that, if $\overline{\vp}_{\omega}(a,a)=0$, for every $\omega\in\mathcal{R}_c(\A,\Ao)$, then $a=0$.
We have
$$|\omega(a)|=|\overline{\vp}_{\omega}(a,e)|\leq \overline{\vp}_{\omega}(e,e)^{1/2}\overline{\vp}_{\omega}(a,a)^{1/2}=0$$
Then using condition (P) we get the statement.

(iii)$\Rightarrow$(i): This can be proved with the same argument used in \cite[Ex. 4.4]{Bag2}. 
\end{proof}

\section{Hilbert quasi *-algebras}\label{Sec4}
Among Banach quasi *-algebras, a distinguished role is played by the completion of a Hilbert algebra with respect to the norm defined by its inner product. This object will be called a {\em Hilbert quasi *-algebra}.
We start with recalling the definition of {Hilbert algebra} (see, e.g. \cite[Section 11.7]{Palmer}).

\bedefi \label{ex_Hilbert algebras}A Hilbert algebra  is a *-algebra
$\Ao$ which is also a pre-Hilbert space with inner product
$\ip{\cdot}{\cdot}$ such that
\begin{itemize}
\item[(i)]The map $y\mapsto xy$ is continuous with respect to the norm
defined by the inner product.
\item[(ii)] $\ip{xy}{z}=\ip{y}{x^*z}$ for all $x,y,z \in \Ao.$
\item[(iii)]
$\ip{x}{y}=\ip{y^*}{x^*}$  for all $x,y \in \Ao$.
\item[(iv)]$\Ao^2$ is total in $\Ao$.
\end{itemize}
\findefi
\berem From (ii) and (iii) it follows also that
$$\ip{xy}{z}=\ip{x}{zy^*}, \quad \forall x,y,z \in \Ao.$$
\enrem
Let $\H$ denote the Hilbert space completion of $\Ao$ with respect to the norm defined by the inner product. The involution of $\Ao$ extends to the whole of $\H$, since (iii) implies that $^*$ is isometric. The multiplication $\xi x$ (or $x\xi$) of an element $\xi$ of $\H$ and an element $x \in \Ao$ is defined by an obvious limit procedure. To avoid trivial situations, we assume that
\begin{itemize}\item[(A)]\; If $\xi \in \H$ and $\xi x=0$, for every $x \in \Ao$, then $\xi=0$.\end{itemize} With these operations, $(\H, \Ao)$ is a Banach
quasi *-algebra which we name {\em Hilbert quasi *-algebra}.

By the definition itself, a Hilbert quasi *-algebra $(\H,\Ao)$ is *-semisimple and, by Theorem \ref{thm_fullrep_semis}, it is fully representable.

As in Definition \ref{defn_219}, we consider, for every $\xi \in \H$, the following operators:
$$ L_\xi: \Ao \to \H, \; L_\xi x= \xi x$$
and

$$ R_\xi: \Ao \to \H, \; R_\xi x= x\xi .$$

\belem \label{lemma_21} Every Hilbert quasi *-algebra is fully closable; i.e., for every $\xi \in \H$, the operators $L_\xi$, $R_\xi$ are closable and, in addition, $L_{\xi^*}\subset (L_\xi)^*$, $R_{\xi^*}\subset (R_\xi)^*$. Hence, $L_\xi \in \LAoH$, for every $\xi\in \H$. Moreover, the map
$$ \xi \in \H \to L_\xi \in \LAoH$$ is injective and, if $\eta \in D(\overline{L}_\xi)$ then $\eta x\in D(\overline{L}_\xi)$, for every $x\in \Ao$.
\enlem
\begin{proof}
It is enough to show that $L_{\xi^*}\subset (L_\xi)^*$, so that $L_\xi$ has a densely defined adjoint. Indeed,
if $x,y\in\Ao$, we have
$$\ip{L_{\xi}x}{y}=\ip{\xi x}{y}=\ip{x}{\xi^{\ast}y}=\ip{x}{L_{\xi^{\ast}}y}.$$
This proves the inclusion $L_{\xi^*}\subset (L_\xi)^*$.
An analogous proof can be done for $R_{\xi}$, $\xi \in \H$.

The injectivity of the map $\xi \in \H \to L_\xi \in \LAoH$ comes from the condition (A).
\end{proof}

\subsection{Partial multiplication and bounded elements}

\bedefi Let $\xi, \eta \in \H$. We say that $\xi$ is a left-multiplier of $\eta$ (or, $\eta$ is a right-multiplier of $\xi$) 
if there exists $\zeta \in \H$ such that
$$ \ip{\eta x}{\xi^* y}= \ip{\zeta x}{y}, \quad \forall x, y \in \Ao.$$
In this case we put $\xi \wmult \eta:=\zeta$. \findefi

Clearly, $\xi \wmult \eta$ is well-defined if and only if $\eta^*\wmult \xi^*$ is well defined. Moreover $(\xi \wmult \eta)^*=\eta^*\wmult \xi^*$.

If $\xi \wmult \eta$ is well-defined then $L_\xi \mult L_\eta$ is well defined in $\LAoH$ and $L_{\xi \wmult \eta}=L_\xi \mult L_\eta$.

For $\xi \in \H$, we denote by ${\sf L}(\xi)$ (resp., ${\sf R}(\xi)$) the set of left- (resp., right-) multipliers of $\xi$.

\begin{prop} Let $(\H,\Ao)$ be a Hilbert quasi *-algebra. Then $\H$ is a partial *-algebra with respect to the multiplication $\wmult$.

\end{prop}


\berem It is clear that $\xi \wmult \eta$ is well defined if and only if $L_\xi \mult L_\eta$ is well defined in $\LAoH$ and $L_\xi \mult L_\eta= L_\zeta$ for some $\zeta \in \H$.
If $(\H, \Ao)$ has a unit $e$, the second condition is automatically satisfied whenever $L_\xi \mult L_\eta$ is well defined, since one can put $\zeta=(L_\xi \mult L_\eta)e$. In this case, $L_\H:=\{L_\xi; \, \xi \in \H\}$ is a partial O*-algebra on $\Ao$.

\enrem

\beprop \label{prop_domainadj} Let $(\H,\Ao)$ be a Hilbert quasi *-algebra and $\xi \in \H$.
The following statements hold.
\begin{itemize}
\item[(i)] ${\sf R}(\xi^*)=\{\eta \in D((L_\xi\upharpoonright_{\Ao^2})^*): \eta y \in D(L_\xi^*),\, \forall y \in \Ao \}.$
\item[(ii)]If $(\H, \Ao)$ has a unit, then ${\sf R}(\xi^*)= D(L_\xi^*)$.
\item[(iii)] If $(\H, \Ao)$ has a unit, and $\xi^*$ is a universal right multiplier ( i.e. ${\sf R}(\xi^*)=\H$), then $\overline{L}_\xi$ and $\overline{L}_{\xi^*}$ are bounded operators.
\end{itemize}
\enprop
\begin{proof} (i): Let $\eta \in {\sf R}(\xi^*)$. Then,
$$\ip{L_\xi x}{\eta y}=\ip{\xi x}{\eta y}= \ip{x}{(\xi^*\wmult \eta) y}.$$
Hence $\eta y \in D(L_\xi^*)$ and $L_\xi^* (\eta y)= (\xi^*\wmult \eta)y$, for every $y \in \Ao$.

Moreover, since
$$\ip{\xi x}{\eta y} = \ip{\xi xy^*}{\eta}= \ip{x}{(\xi^*\wmult \eta) y}=\ip{xy^*}{\xi^*\wmult \eta},$$
we also have that $\eta \in D((L_\xi\upharpoonright_{\Ao^2})^*)$.

Conversely, let $\eta \in D((L_\xi\upharpoonright_{\Ao^2})^*)$ be such that $\eta y \in D(L_\xi^*)$, for every $y \in \Ao$.
Then,
$$ \ip{L_\xi x}{\eta y}= \ip{x}{L_\xi^* (\eta y)}, \quad \forall x, y \in \Ao$$
On the other hand,
$$  \ip{L_\xi x}{\eta y}=  \ip{L_\xi xy^*}{\eta} =  \ip{ xy^*}{(L_\xi\upharpoonright_{\Ao^2})^*\eta }=\ip{ x}{((L_\xi\upharpoonright_{\Ao^2})^*\eta)y }.$$
Thus,
$$\ip{\xi x}{\eta y}=\ip{ x}{((L_\xi\upharpoonright_{\Ao^2})^*\eta)y }.$$
This implies that $\eta \in {\sf R}(\xi^*)$ and $\xi^*\wmult \eta = (L_\xi\upharpoonright_{\Ao^2})^*\eta$.\\
(ii): This follows immediately from the closed graph theorem.\\
(iii): In this case, $\Ao^2 =\Ao$. Now if $\eta \in D(L_\xi^*)$ and $y \in \Ao$, then $\eta y\in D(L_\xi^*) $. Indeed, we have
$$ \ip{L_\xi x}{\eta y}=\ip{L_\xi (xy^*)}{\eta}=\ip{ xy^*}{L_\xi^*\eta}= \ip{ x}{(L_\xi^*\eta)y}.$$
Hence, $\eta y\in D(L_\xi^*) $ and $L_\xi^*(\eta y)= (L_\xi^*\eta)y.$
\end{proof}

%
%
%

For some $\xi\in \H$ it may happen that the operator $L_\xi$ (resp. $R_\xi$) is bounded on $\Ao$. Then its closure $\overline{L}_\xi$ (resp. $\overline{R}_\xi$) is an everywhere defined bounded operator in $\H$. In this case, we say that $\xi$ is a {\em left- ({\rm resp.} right-)  bounded element}.

\medskip
The following result holds \cite[Proposition 11.7.5]{Palmer}:
\begin{prop}\label{defn_bounded}Let $(\H, \Ao)$ be a Hilbert quasi *-algebra. An element $\xi \in \H$ is left-bounded if and only if it is right bounded. Moreover, $\overline{L}_\xi x= (R_{\xi^*}x^*)^*$, for every left-bounded element $\xi\in \H$.
\end{prop}
From now on, we speak only of {\em bounded elements}. The set of bounded elements is denoted by $\H_b$.

 Let $\xi,\eta \in\H_b$. Then the multiplication $\xi \wmult \eta$ is, in this case, well-defined.
Indeed, for every $x,y \in \H$, we have
\begin{align*}\ip{\eta x}{\xi^* y}&= \ip{L_\eta x}{\xi^* y}=\ip{ x}{(L_\eta)^*(\xi^* y)}=\ip{ x}{\overline{L}_{\eta^*}(L_{\xi^*} y)} \\&= \ip{ x}{(\overline{L}_{\eta^*}\overline{L}_{\xi^*}) y)}=\ip{(\overline{L}_\xi\eta) x}{ y)} .\end{align*}
Hence, $\xi \wmult \eta=\overline{L}_\xi\eta$.
Thus $\H_{b}$ is a *-algebra containing $\Ao$. It is natural to define a norm in $\H_{b}$ by
$ \|\xi\|_{b} = \|L_\xi\|$ where the latter denotes the operator norm of ${\mc B}(\H)$. We notice that there is no ambiguity in this choice, because, as it is easy to check, $\|L_\xi\|=\|R_\xi\|$, for every bounded element $\xi$.

\begin{prop} The space $\H_{b}$ is a pre C*-algebra. If, in addition, $(\H,\Ao)$ has a unit $e$, then $\H_{b}[\|\cdot\|_{b}]$ is complete and, therefore, it is a C*-algebra.
\end{prop}

From (iii) of Proposition \ref{prop_domainadj} we get
\begin{prop} If $(\H,\Ao)$ has a unit $e$, then the space $\H_b$ of bounded elements coincides with the set ${\sf R}\H(={\sf L}\H)$ of the universal multipliers of $\H$.
\end{prop}

\berem If $(\H,\Ao)$ has no unit, then $\H_{b}[\|\cdot\|_{b}]$ need not be complete. As an example take $\Ao =C_c({\mb R})$ the algebra of continuous functions with compact support with the $L^2-$ inner product. Then $\H=L^2({\mb R}, d\lambda)$, where $\lambda$ indicates the Lebesgue measure on the real line. In this case $\H_{b}= L^\infty({\mb R}, d\lambda) \cap L^2({\mb R}, d\lambda)$ which is not complete in the $L^\infty$-norm. 
\enrem

\berem The notion of bounded element of a Banach quasi *-algebra $(\A,\Ao)$ was first given in \cite{ct1, ct2}. It is coincides exactly with that given here. Bounded elements were also generalized to different situations in \cite{Ant3, bell_ct}. 
\enrem
\subsection{Positive elements and representable functionals} As done in Section \ref{Sec2}, we consider the set $\H^+=\overline{\Ao^+}^{\tau_n}$ of positive elements of $\H$, where $\Ao^+:=\left\{\sum_{k=1}^nx_k^{\ast}x_k:x_k\in\Ao,n\in\mathbb{N}\right\}$.

\bedefi\label{pos} Let $(\H,\Ao)$ be a Hilbert quasi *-algebra and $\xi\in \H$.  We say that $\eta$ is {\em w-positive} if $L_\eta$ (or, equivalently $R_\xi$)  is a positive operator; i.e., if
$$\ip{L_\xi x}{x}=\ip{\xi x}{x}\geq 0, \forall x \in \Ao.$$ \findefi

If $\xi$ is w-positive, then $\xi=\xi^*$. Moreover $L_\xi$ is positive, if and only if, $R_\xi$ is positive.

The latter statement is due to the equalities
$$ \ip{L_\xi x}{x} = \ip{\xi x}{x}= \ip{x^*}{x^*\xi^*}\geq 0  \mbox{ if } L_\xi \geq 0 \mbox{ and } x \in \Ao.$$

We put
$$\H^+_w=\{ \xi, \in \H:\, \ip{\xi x}{x}\geq 0, \forall x \in \Ao \}= \{ \xi, \in \H:\, \xi \mbox { is w-positive}\}.$$

The wedge $\H^+_w$ defines in standard way  a partial order in the real space $\H_h= \{\xi \in \H: \xi=\xi^*\}$: if $\xi, \eta \in \H_h$ we write $\xi \leq \eta$ when $\eta -\xi \in \H^+_w$.
The set $\H^+_w$ is actually a cone. Indeed, it is easy to prove that
$\H^+_w \cap (-\H^+_w)= \{0\}$.

\medskip
The notion of w-positive element plays a role in the description of representable and continuous functionals on $(\H, \Ao)$.


To begin with, we describe the general form of an element $\omega\in \hcrep$.
Let $\omega \in \hcrep$.
 Since $\omega$ is bounded there exists a vector $\eta \in \H$ such that
\begin{equation} \label{eqn_omega}\omega(\xi)= \ip{\xi}{\eta}, \quad \forall \xi \in \H.\end{equation}

The condition (L1) implies that $\eta$ is w-positive and hence the operator $R_\eta$ is positive (but not necessarily self-adjoint). Indeed,
$$ 0\leq \omega(x^*x) = \ip{x^*x}{\eta} = \ip{x}{x\eta}= \ip{x}{R_\eta x},\quad \forall x\in \Ao.$$
This in turn implies that $\eta =\eta^*$.

As for condition  (L2), we have, for every $\xi\in \H$ and $x, y \in \Ao$,
$$ \omega(y^* \xi x)= \ip{y^* \xi x}{\eta} = \ip{\eta^*}{x^* \xi^* y}= \ip{\eta}{x^* \xi^* y}= \overline{\ip{x^* \xi^* y}{\eta}}= \overline{\omega(x^*\xi^* y)}$$

\medskip
The condition (L3) in the situation we are examining reads as follows
$$ \forall \xi \in \H, \, \exists \gamma_\xi>0:\, |\ip{\xi^* x}{\eta}|\leq \gamma_\xi \ip{x^*x}{\eta}^{1/2}, \; \forall x \in \Ao.$$
{ From proposition \ref{prop1}, $\varphi_{\omega}(x,x)=\ip{x^{\ast}x}{\eta}$ has an everywhere defined closure $\overline{\varphi}_\omega$ in $\Hil$, hence for every $\xi\in\H$ there exists $c_{\xi,\eta}>0$ such that
$$  |\ip{\xi^* x}{\eta}|\leq c_{\xi,\eta} \|x\|, \quad \forall x\in \Ao.$$}
This implies that there exists a vector $\eta' \in \H$ such that
$$ \ip{\xi^* x}{\eta} = \ip{x}{\eta'}, \quad \forall x\in \Ao.$$
Hence, for every $x, y \in \Ao$ we get
$$ \ip{\xi^* x}{\eta y}= \ip{\xi^* x y^*}{\eta}= \ip{x y^*}{\eta'}=\ip{x}{\eta' y}.$$
Therefore, $\xi\wmult \eta$ is well defined for every $\xi \in \H$. Thus ${\sf L}(\eta)=\H$ and, since $\eta=\eta^*$, ${\sf R}(\eta)=\H$.
Taking into account Lemma \ref{lemma_21}, both $L_\eta$ and $R_\eta$ are bounded operators on $\H$.

In conclusion we have the following
\beprop \label{prop_repfun}  Let $(\H,\Ao)$ be a Hilbert quasi *-algebra. Then, $\omega \in \hcrep$ if, and only if, there exists a unique w-positive bounded element $\eta\in \H$ such that
\begin{equation}\label{eqn_repfun} \omega(\xi)=\ip{\xi}{\eta}, \quad \forall \xi \in \H.\end{equation}
\enprop
\begin{proof}The necessity has been proved above. For the sufficiency, (L1) and (L2) can be proven easily.
As for (L3), denoting by $\overline{R}_\eta$ the continuous extension of $R_\eta$ to $\H$, we have, for every $\xi \in \H$,
\begin{align*}|\omega(\xi^* x)| &= |\ip{\xi^* x}{\eta}|= |\ip{ x}{\overline{R}_\eta\xi}|\\
&\leq \ip{ x}{R_\eta x}^{1/2} \ip{ \xi}{\overline{R}_\eta\xi}^{1/2}= \omega(x^*x)^{1/2} \ip{ \xi}{\overline{R}_\eta\xi}^{1/2}, \quad \forall x \in \Ao \end{align*}
due to the generalized Cauchy-Schwarz inequality for positive operators. Thus (L3) is fulfilled. \end{proof}
{\berem Let $(\H,\Ao)$ be a Hilbert quasi *-algebra. We already know from Proposition \ref{prop_charuniformly_rep} that, $\hcrep=\hcurep$.
Here we want to point out that, in this case, this result can be obtained by a direct estimation.

Indeed, if $\omega\in \hcrep$, then, by Proposition \ref{prop_repfun}, there exists $\eta\in \H_b$, $\eta$ w-positive, such that \eqref{eqn_repfun} holds. Then, if $\xi\in \H$ and $x \in \Ao$, we have
\begin{align*} |\omega(\xi^*x)|&=|\ip{\xi^*x}{\eta}|=|\ip{\xi^*}{\overline{L}_\eta x^*}|\\
&= |\ip{\overline{L}_\eta^{1/2} \xi^*}{\overline{L}_\eta^{1/2} x^*}|\leq \|\overline{L}_\eta^{1/2} \xi^*\| \|\overline{L}_\eta^{1/2} x^*\|\\
& =\ip{\xi^*}{\overline{L}_\eta\xi^*} \ip{x^*}{\overline{L}_\eta x^*}\leq \|\overline{L}_\eta^{1/2}\|\,\|\xi\|\omega(x^*x)^{1/2}.
\end{align*}
Hence, $\omega$ is uniformly representable.
\enrem
}
{
After describing the set $\crep$, a second natural step consists in looking for conditions on $\omega \in { \mc R} (\H, \Ao)$ to be continuous on $\H$, at least in some particular situations.

Before proceeding, we remind that a {\it critical eigenvalue} of a couple of operators $A,B$ in a Banach space $X$ is a complex number $\mu$ such that $(A-\mu I)X$ has infinite codimension and $\mu$ is an eigenvalue of $B$ (see \cite[Def 3.1]{Sin}.

\begin{prop}\label{comm_HQA} Let $(\H,\Ao)$ be a commutative Hilbert quasi *-algebra with unit $e$.
Assume that $\Ao$ is a Banach *-algebra with respect to $\|\cdot\|_0$ and that there exists an element $x$ of $\Ao$ such that the spectrum $\sigma(\overline{R}_x)$ of the bounded operator $\overline{R}_x$ of right multiplication by $x$ consists only of its continuous part $\sigma_c(\overline{R}_x)$.
If $\omega\in \rep$,
then $\omega$ is bounded (and uniformly representable).
\end{prop}

\begin{proof}
We maintain the notations of Proposition \ref{prop_operatorT} and discussion around it.
Then, if we put
$$\lambda_\omega(x)\cdot \lambda_\omega(y) = \lambda_\omega(xy),$$ it is easily seen that
$\cdot$ is a well-defined multiplication on $\lambda_\omega (\Ao)$ which makes it into a commutative *-algebra.
Moreover, if $\Ao$ is a Banach *-algebra with respect to $\|\cdot\|_0$, we have
\begin{equation}\label{eqn_boundR} \|\lambda_\omega(xy)\|^2\leq  \|y\|_0 \|\lambda_\omega(x)\|^2.\end{equation}
The inequality \eqref{eqn_boundR} implies that the right multiplication operator $R_{\lambda_\omega(y)}$ by $\lambda_\omega(y)$ is bounded on $\lambda_\omega(\Ao)$ and therefore it has a unique bounded extension to $\H_\omega$ denoted by the same symbol.
Taking all that into account, one easily checks that
\begin{align}
&T_\omega(az)= T_\omega a\cdot \lambda_\omega(z), \; \forall a \in \A, z\in \Ao;\label{eqn_first}\\
&T_\omega z= \lambda_\omega(z),\; \forall  z\in \Ao.
\end{align}
In particular \eqref{eqn_first} reads as follows
\begin{equation}\label{eqn_intertwining} T_\omega(R_za)=R_{\lambda_\omega(z)}T_\omega a , \; \forall a \in \A, z\in \Ao;\end{equation}
i.e., $T_\omega$ intertwines the couple $(R_z, R_{\lambda_\omega(z)})$, for every $z \in \Ao$.
Then the continuity of $T_\omega$ can be deduced from \cite[Corollary 5.7]{Sin}: if there exists $z \in \Ao$ such that
$R_z, R_{\lambda_\omega(z)}$ are normal (which is obvious here) and have no critical eigenvalues, then $T_\omega$ is continuous.

By the assumption, there exists an element $x\in \Ao$ such that $\sigma(\overline{R}_x)=\sigma_c(\overline{R}_x)$. Hence, for every $\mu \in {\mb C}$, the range of the operator $R_x -\mu I$ is either $\H$ itself or a dense subspace of $\H$. Thus no critical eigenvalue for the couple $(R_x, R_{\lambda_\omega(z)})$ may exists.
\end{proof}}

Proposition \ref{prop_repfun} allows to compare
the two sets of elements $\H^+$ and $\H^+_w$ related to the notion of positivity.

\begin{lemma} Let $(\H,\Ao)$ be a Hilbert quasi *-algebra. Then, $\H^+ \subseteq \H^+_w$.
\end{lemma}
\begin{proof} Let $\eta\in \H^+$; i.e. there exists a sequence $\{x_n\}$ of elements of $\Ao^+$ (thus each $x_n$ has the form $x_n = \sum_{k=1}^{N} z_{nk}^*z_{nk}$, $z_{nk}\in \Ao$) $\tau_n-$converging to $\eta$.
 Then,
\begin{align*} \ip{L_\eta y}{y}&= \ip{\eta y}{y}= \lim_{n \to \infty}\ip{x_n y}{y} \\ &= \lim_{n \to \infty}\ip{\sum_{k=1}^{N} z_{nk}^*z_{nk} y}{y}=\lim_{n \to \infty}\sum_{k=1}^{N}\ip{z_{nk}y}{z_{nk}y}\geq 0.\end{align*}
\end{proof}
\beprop Let $(\H,\Ao)$ be a Hilbert quasi *-algebra.
If the condition (P) holds and $\H_{wb}^+:=\{\xi \in \H_b:\, \ip{\xi x}{x}\geq 0,\, \forall x\in \Ao\}\subseteq \H^+$, then $\H^+ =\H^+_w$. \end{prop}
\begin{proof}

Suppose now that the condition (P) is valid and assume that $\eta \in \H^+_w$. Then $\ip{\eta y}{y}\geq 0$ for every $y \in \Ao$. This implies that $\eta\in\H$ is w-positive if, and only if, $\langle\eta,h\rangle\geq0$ for every $h\in\Ao^+$, so for continuity, $\langle\eta,\xi\rangle\geq0$ for every $\xi\in\H^+$.

Let $\omega \in { \mc R}_c (\H, \Ao)$. By Proposition \ref{prop_repfun}, there exists $\chi_{\omega}\in\H$, with $\chi_{\omega}\in\H$ bounded and w-positive, such that
 $ \omega(\xi)=\ip{\xi}{\chi_\omega}$, for every $\xi \in \H$. By the assumption $\chi_\omega = \displaystyle\lim_{n\to \infty} z_n$ for certain $z_n$'s in $\Ao^+$. Hence, for what we have just observed,
 $$\omega(\eta )=\ip{\eta }{\chi_\omega}=\lim_{n\to \infty}\ip{\eta}{z_n}\geq0.$$
Since $\omega \in { \mc R}_c (\H, \Ao)$, then  $\omega_x \in { \mc R}_c (\H, \Ao)$, where $\omega_x(\zeta)= \omega(x^*\zeta x),\; \zeta \in \H,\, x\in \Ao$. By condition (P) it follows that $\eta \in \H^+$.\end{proof}

\subsection{Integrable Hilbert quasi *-algebras}

Let $(\H,\Ao)$ be a Hilbert quasi *-algebra.
{As we have seen, in general, if $\xi \in \H$ and $\xi$ is not bounded, we have $\overline{L_\xi}\subseteq L_{\xi^*}^*$, but we do not know if the equality holds for every $\xi\in \H$.
For this reason we introduce the following definition.

\bedefi \label{defn_integrable} A Hilbert quasi *-algebra $(\H, \Ao)$ is called {\em integrable} if $\Ao$ is a core for  $L_\xi^*$, for every $\xi \in \H$.
\findefi

It is clear that, if $(\H, \Ao)$ is integrable, then, for every $\xi \in \H$, $\overline{L_\xi}= L_{\xi^*}^*$, since $L_{\xi^*}\subseteq L_\xi^*$. In particular, if $\xi=\xi^*$, then $L_\xi$ is essentially self-adjoint.
}


\bedefi \label{defn_module} Let $(\H,\Ao)$ be a Hilbert quasi *-algebra with unit $e$. We say that $(\H,\Ao)$ admits a {\em module function} if
there exists a map $\mu:\xi\in \H \to \mu(\xi)\in \H$, with the following properties:
\begin{itemize}
\item[(i)]$\mu(\xi)\in \H^+_w$, for every $\xi\in \H$;
\item[(ii)]$\mu(\xi)=\xi£$, for every $\H^+_w$;
\item[(iii)]$\|\mu(\xi)\|=\|\xi\|$, for every $\xi\in \H$.
\end{itemize}
\findefi

\begin{prop} Let $(\H,\Ao)$ be an integrable Hilbert quasi *-algebra with unit $e$. Assume, that for every $\xi \in \H$, the operator $H_{(\xi)}$ defined by $H_{(\xi)}= (L_\xi^* \overline{L}_\xi)^{1/2}$ has the following property:
\begin{equation}\label{eqn_47}H_{(\xi)}(xy)=(H_{(\xi)}x)y \quad \forall x,y \in \Ao.\end{equation}
Then $(\H,\Ao)$ admits a module function.
\end{prop}

\begin{proof} As it is well-known, $H_{(\xi)}$ is a positive self-adjoint operator with $D(H_{(\xi)})= D(\overline{L}_\xi).$
By choosing $x=e$ in \eqref{eqn_47}, we get $H_{(\xi)}y= (H_{(\xi)}e)y$ for every $y \in \Ao$; that is, $H_{(\xi)}y=L_{(H_{(\xi)e})}y$, for every $y\in \Ao$.

The module function is then defined everywhere in $\H$ by the following map $$\mu: \xi \in \H \to \mu(\xi):=H_{(\xi)}e \in \H.$$
We check the conditions (i)-(iii) of Definition \ref{defn_module}.\\
We have
$$ \ip{\mu(\xi)x}{x} = \ip{(H_{(\xi)}e) x}{x} = \ip{H_{(\xi)}x}{x} \geq 0, \quad \forall x \in \Ao$$
Hence $\mu(\xi)\in \H^+_w$ . Moreover, $\mu(\xi)=0$ if and only if $\xi=0$.

As for (ii), for every $\xi\in \H^+_w$, the operator $L_\xi$ is essentially self-adjoint and positive. Hence  $H_{(\xi)}= ((L_\xi)^*\overline{L}_\xi)^{1/2}=\overline{L}_\xi$. This implies that  $\mu(\xi)=\xi$.

Finally, for every $\xi \in \H$,
 $$ \|\xi\|^2 =\ip{\xi}{\xi}=\ip{\xi e}{\xi e}=\ip{L_\xi e}{L_\xi e}=\ip{H_{(\xi)} e}{H_{(\xi)} e}= \|\, \mu(\xi)\, \|^2.$$
So (iii) holds.
\end{proof}
\berem Condition (iii) of Definition \ref{defn_module} obviously implies  that $\mu$ is continuous at $0$. \enrem
\begin{lemma}\label{mod} Let $(\H,\Ao)$ be a unital integrable Hilbert quasi *-algebra with a module function $\mu$. Then,
for every $\xi \in \H$, with $\xi=\xi^*$, there exists two elements $\xi_+, \xi_-\in \H^+_w$ such that
$\xi=\xi_+-\xi_-$; $\mu(\xi)=\xi_++\xi_-$; $\xi_+\wmult \xi_-$ and $\xi_- \wmult \xi_+$ are both well-defined and $\xi_+\wmult \xi_-=\xi_- \wmult \xi_+=0$.
\end{lemma}
\begin{proof}

(ii): Let us define the following two operators on $D(\overline{L}_\xi)=D(H_{(\xi)})$:
$$ P_{(\xi)}:= \frac{1}{2}(H_{(\xi)}+\overline{L}_\xi), \qquad N_{(\xi)}:= \frac{1}{2}(H_{(\xi)}-\overline{L}_\xi).$$
Put $\xi_+ = P_{(\xi)}e$ and $\xi_- =N_{(\xi)}e$. These elements are in $\Hil_w^+$, indeed for every $x\in\Ao$ we have
$$2\ip{\xi_+x}{x}=\ip{(H_{(\xi)}e+\overline{L}_{\xi}e)x}{x}=\ip{\mu(\xi)x}{x}+\ip{\xi x}{x}\geq0.$$
Since the wedge of w-positive elements $\Hil^+_w$ is in fact a cone, the worst case happens when $\ip{\xi x}{x}\leq0$ for every $x\in\Ao$, i.e. when $-\xi\in\Hil_w^+$, but $\mu(\xi)=\mu(-\xi)=-\xi$.
The proof for $\xi_-$ is similar.

It is clear that $\xi=\xi_+-\xi_-$ and $\mu(\xi)=\xi_++\xi_-$.

Let $x,y\in\Ao$. We want to show that $\xi_+\wmult\xi_-=0$, i.e.
$$\ip{(H_{(\xi)}e-\overline{L}_{\xi}e)x}{(H_{(\xi)}e+\overline{L}_{\xi}e)y}=0.$$
Computing in the expression above, we obtain
\begin{align*}\ip{(H_{(\xi)}e)x}{(H_{(\xi)}e)y}&+\ip{(H_{(\xi)}e)x}{\overline{L}_{\xi}y}-\ip{\overline{L}_{\xi}x}{(H_{(\xi)}e)y}-\ip{\overline{L}_{\xi}x}{\overline{L}_{\xi}y}\\
=&\ip{\mu(\xi)x}{\mu(\xi)y}+\ip{\mu(\xi)x}{\xi y}-\ip{\xi x}{\mu(\xi)y}-\ip{\xi x}{\xi y}
\end{align*}
Now, we can show that $\|\mu(\xi)x\|=\|\xi x\|$, for every $x\in\Ao$. Indeed,
\begin{align*}
\|\xi x\|^2 &=\ip{\xi x}{\xi x}=\ip{L_{\xi}x}{L_{\xi}x}\\
&=\ip{H_{(\xi)}x}{H_{(\xi)}x}=\ip{(H_{(\xi)}e)x}{(H_{(\xi)}e)x}\\
&=\|\, \mu(\xi)x\, \|^2, \quad \forall \xi \in \H.
\end{align*}
By the polarization formula we have $\ip{\mu(\xi)x}{\mu(\xi)y}=\ip{\xi x}{\xi y}$ for every $x,y\in\Ao$.

Notice that, by (ii) of Definition \ref{defn_module}, $\mu^2=\mu$, i.e. $\mu\left(\mu(\xi)\right)=\mu(\xi)$ for every $\xi\in\Hil$. Hence, if $\mu^2(\xi):=\mu\left(\mu(\xi)\right)$,
$$\ip{\mu(\xi)x}{\xi y}=\ip{\mu^2(\xi)x}{\mu(\xi)y}=\ip{\mu(\xi)x}{\mu^2(\xi)y}=\ip{\xi x}{\mu(\xi)y}.$$

For what we have just shown, we have that $\xi_-\wmult\xi_+$ is well-defined and $\xi_-\wmult\xi_+=0$. Using the same argument, it is easily proven that $\xi_+\wmult\xi_-=0$. Then $\xi_+$ and $\xi_-$ have the desired properties.
\end{proof}

\begin{prop}\label{mod_1} Let $(\H,\Ao)$ be a Hilbert quasi *-algebra with a continuous module function. Suppose that $\mu(\Ao) \subseteq \Ao^+$.
 Then every positive element is the limit of a sequence of elements of $\Ao^+$; i.e., $\H^+=\H^+_w$.
\end{prop}
\begin{proof}
Let $\xi \in \H^+_w$ then by density there exists a sequence $\{x_n\} \subset \Ao$ such that $x_n \to \xi$. By the assumptions $\mu(x_n)\in \Ao^+$ and $\mu(x_n)\to \mu(\xi)=\xi$.
\end{proof}

{ \begin{prop} \label{prop_4.27}Let $(\H,\Ao)$ be a unital integrable Hilbert quasi *-algebra with a $\tau_n-$continuous module function. Suppose that $\mu(\Ao) \subseteq \Ao^+$.
Every $\omega\in \rep$ which is positive on $\H^+_w$ is continuous.
\end{prop}
\begin{proof} As shown in \cite{Trap1}, every positive functional on $\H^+_w=\H^+$ (by Proposition \ref{mod_1}) is bounded on positive elements. The statement then follows from Lemma \ref{mod}.\end{proof}}

{\section{Examples: Function quasi *-algebras}\label{Sec5}}
Let $I$ be a compact interval of the real line and  $\lambda$ the Lebesgue measure on it.
In this section we will show that every representable functional over $\left(L^2(I,d\lambda),\Ao\right)$, where $\Ao= C(I)$ or $\Ao=L^{\infty}(I,d\lambda)$, is continuous.

Both $\left(L^2(I,d\lambda),C(I)\right)$ and $\left(L^2(I,d\lambda),L^{\infty}(I,d\lambda)\right)$ are Hilbert quasi *-algebras and, as it easy to see, both are integrable in the sense of Definition \ref{defn_integrable}. A description of representable functionals on $\left(L^2(I,d\lambda),C(I)\right)$ is provided
by the following representation theorem.

\begin{prop}\label{prop_52}
Let $\omega$ be a representable functional on $(L^2(I,d\lambda), C(I))$. Then there exists a {unique} Borel measure $\mu$ on $I$ and a {unique} bounded linear operator $T: L^2(I, d\lambda)\to L^2(I, d\mu)$ such that
\begin{equation}\label{eqn_int_omega} \omega(f) = \int_I (Tf)d\mu, \quad \forall f \in L^2(I, \lambda).\end{equation}
The operator $T$ satisfies the following conditions:
\begin{align}
&T(f\phi)=(Tf)\phi=\phi(Tf)\quad \forall f\in L^2(I,d\lambda),\phi\in C(I); \label{eqn_1.2} \\
&T\phi=\phi, \quad \forall \phi \in C(I). \label{eqn_1.3}
\end{align}
Thus, every representable functional $\omega$ on $\left(L^2(I,d\lambda),C(I)\right)$ is continuous.
Moreover, $\mu$ is absolutely continuous with respect to $\lambda$.
\end{prop}

\begin{proof}
By definition $\omega$ is positive on $C(I)$. Therefore by the Riesz-Markov theorem, there exists a unique Borel measure $\mu$ on $I$ such that
$$ \omega(\phi) =\int_I \phi d\mu, \quad \forall \phi \in C(I).$$
By condition (L.3), for every $f \in L^2(I, d\lambda)$, there exists $\gamma_f$ such that
$$|\omega(f^*\phi)| \leq \gamma_f \|\phi\|_{2,\mu}, \quad \forall \phi \in C(I).$$
Hence, the linear functional $L_f$ defined by $L_f(\phi)=\omega(f^*\phi)$, $\phi \in C(I)$, is bounded on $C(I)$, with respect to  $\|\cdot\|_{2,\mu}$. Thus, it extends to a bounded linear functional on $L^2(I,d\mu)$ and there exists a unique function $h_f\in L^2(I,d\mu)$ such that
\begin{equation}\label{eqn_L3} \omega(f^*\phi)=\int_I \phi \overline{h_f}d\mu, \quad \forall \phi \in C(I).\end{equation}
We can define a linear map $T: L^2(I, d\lambda) \to L^2(I,d\mu)$ by putting $Tf=h_f$, $f \in L^2(I, d\lambda)$. With this definition we have
$$ \omega(f) = \int_I (Tf)d\mu, \quad \forall f \in L^2(I, d\lambda).$$
The equalities \eqref{eqn_1.2} and \eqref{eqn_1.3} are easily proved.
In particular, \eqref{eqn_1.2} can be rewritten as follows
$$TR_{\phi}f=R'_{\phi}Tf,\quad\forall f\in L^2(I,\lambda), \phi\in C(I), $$
where $R_{\phi}$ and $R'_{\phi}$ denote the multiplication operators by $\phi$ in $L^2(I,d\lambda)$ and $L^2(I,d\mu)$, respectively.
This means that $T$ intertwines the couple $(R_\phi,R'_\phi)$ for every  $\phi\in C(I)$. 
The operator $T$ is continuous if, and only if, there exists $\phi \in C(I)$ such that the couple $(R_\phi,R'_\phi)$ has no critical eigenvalues (again by \cite[Theorem 5.5]{Sin}).

In $L^2(I,d\lambda)$, ($\lambda$ the Lebesgue measure) the operator $R_\phi$ has continuous spectrum. Hence, the statement follows from Proposition \ref{comm_HQA}.


Since $T$ is bounded, it has an adjoint $T^*: L^2(I, d\mu) \to L^2(I, d\lambda)$. Hence, if we denote by $u$ the unit function in $C(I)$ (i.e., $u(x)=1$, for every $x \in I$), we get
\begin{equation}\label{eqn_abscont} \omega(f)= \int_I Tf d\mu= \int_I (Tf)u d\mu = \int_I f \overline{(T^*u)} d\lambda, \quad \forall f \in L^2(I, d\lambda).\end{equation}
It is easily seen that $T^*u$ is a nonnegative function and by \eqref{eqn_abscont} it follows also that $\omega$ is necessarily positive on positive elements of $L^2(I,d\lambda)$.

Put $w=T^*u$.
From \eqref{eqn_abscont}, we get, in particular,
$$ \omega(\phi)=\int_I \phi d\mu=  \int_I \phi w d\lambda, \quad \forall \phi \in C(I).$$
The previous equality implies by the uniqueness of the measure associated to a positive linear functional on $C(I)$  us that $d\mu=wd\lambda$; i.e., $\mu$ is $\lambda$-absolutely continuous with Radon-Nikodym derivative $w$.
\end{proof}

\medskip
Let us now consider the Banach quasi *-algebra $\left(L^2(I,d\lambda),L^{\infty}(I,d\lambda)\right)$. In this case we have more information about the measure that allows us to represent the functional.

\begin{prop}\label{prop_53} Let $\omega$ be a representable functional on the Banach quasi *-algebra $\left(L^2(I,d\lambda), L^{\infty}(I,d\lambda)\right)$. Then there exists a {unique} {bounded finitely additive measure} $\nu$ on $I$ which vanish on subsets of $I$  of zero $\lambda-$measure and a {unique} bounded linear operator $S: L^2(I, d\lambda)\to L^2(I, d\nu)$ such that
\begin{equation}\label{eqn_int_omega} \omega(f) = \int_I (Sf)d\nu, \quad \forall f \in L^2(I, d\lambda).
\end{equation}
The map $S$ has the following properties:
\begin{align}
&S(f\phi)=(Sf)\phi=\phi(Sf)\quad \forall f\in L^2(I,d\lambda), \phi\in L^{\infty}(I,d\lambda),; \label{eqn_1.6} \\
&S\phi=\phi, \quad \forall \phi \in L^{\infty}(I,d\lambda). \label{eqn_1.7}
\end{align}
\end{prop}
\begin{proof} The proof of this statement is essentially the same as that of Proposition \ref{prop_52}.
Indeed, by \cite[Th. IV.8.16]{DuSc}, there exists a complex valued measure $\nu$ {absolutely continuous} with respect to $\lambda$, for which $\omega$ has the following form
$$\omega(\phi)=\int_I\phi d\nu\qquad\phi\in L^{\infty}(I,d\lambda).$$
Since the functional $\omega$ is positive, i.e. $\omega(\phi^{\ast}\phi)\geq0$, for every $\phi\in L^{\infty}(I,d\lambda)$, the measure $\nu$ is positive.

Following the proof of Proposition \ref{prop_52}, one can show that there exists a unique linear map
$S:L^2(I,d\lambda)\to L^2(I,d\nu)$
such that
$$\omega(f)=\int_I(Sf)d\nu, ,\quad\forall f\in L^2(I,d\lambda).$$
The boundedness of $S$ follows directly from the analogous statement used for $T$ in the proof of Proposition \ref{prop_52}, taking into account that $S$ intertwines the multiplication operators for a function $\phi\in L^\infty(I,\lambda)$. In particular $\phi$ can be chosen to be continuous.
\end{proof}

It is clear that the continuity of the operators $T$ and $S$ implies the continuity of the corresponding representable functionals.
Thus, we can conclude with the following
\begin{cor} Every representable functional over the Banach quasi *-algebras $\left(L^2(I,d\lambda),C(I)\right)$ or $\left(L^2(I,d\lambda),L^{\infty}(I,d\lambda)\right)$, where $I$ is a compact interval of the real line and  $\lambda$ be the Lebesgue measure on it, is continuous.
\end{cor}

\section*{Conclusion remarks}
Representable functionals on locally convex quasi *-algebras play a very important role not only for the understanding of their structure (see \cite{Ant2, bell_ct, Frag2, Lass2}), but also for several applications (see \cite{Bag8, Bag6, Bag7, Lass3, Trap3}). A typical istance is provided (as mentioned in the Introduction) by the mathematical description of certain quantum models (coming from Quantum Statistics and Quantum Field Theory) if the usual approach where the observable are supposed to be in a C*-algebra fails.

In this paper we have investigated the problem of continuity of functionals for which a GNS construction is possible. Even though our answer to this problem is still incomplete (but work is in progress), we thing that the results discussed here contribute to give a better insight on the structure of locally convex (Banach, in particular) quasi *-algebras.

\bigskip
{\bf{Acknowledgement:} }
This work has been done in the framework of the project ''Quasi *-algebre di operatori: propriet\`a spettrali ed applicazioni'', INDAM-GNAMPA 2016.

\end{document}